\documentclass{article}
\usepackage{amsfonts}
\usepackage{amssymb}
\usepackage{amsmath}
\usepackage{amsthm}
\usepackage[dvips]{graphics}

\newtheorem{theorem}{Theorem}
\newtheorem{proposition}{Proposition}
\newtheorem{lemma}[proposition]{Lemma}

\begin{document}
   \title{ Non-intersecting paths and Hahn orthogonal polynomial ensemble}
   \author{Vadim Gorin}
   \date{}

\maketitle

\begin{abstract} We compute the bulk limit of the correlation functions for the uniform measure on lozenge
  tilings of a hexagon. The limiting determinantal process is a translation invariant extension of the
  discrete sine process, which also describes the ergodic Gibbs measure of an appropriate slope.
\end{abstract}

\section*{Introduction}

In this paper we study a well-known combinatorial model of
determinantal point processes in dimension 1+1 (one spatial and one
time variable). The model depends on 3 positive integers $a,b,c$ and
is given by the uniform distribution on the finite set
$\Omega(a,b,c)$ of combinatorial objects that can be described in
several equivalent ways: 3d Young diagrams (in other words plane
partitions) in an $a\times b\times c$ box; tilings of the hexagon of
side lengths $a,b,c,a,b,c$ by rhombi of three types; and a dimer
model on the honeycomb lattice (see e.g. Cohn--Larsen-Propp
\cite{Cohn-Larsen-Propp}, Johansson \cite{Johansson},
Johansson--Nordenstam \cite{JohN}). Correspondences between
different models can be found in the appendix.

Following Johansson, we use yet another model where $\Omega(a,b,c)$
is identified with the ensemble of non-intersecting polygonal paths
on the plane lattice.  The latter model leads to determinantal point
processes (put it otherwise, random point configurations) varying
over time. The determinantal property, which is of crucial
importance, means that dynamical (i.e. space-time) correlation
functions of the model can be obtained as minors of some matrix
which is called dynamical correlation kernel. The aim of the present
paper is to compute the asymptotics of this kernel in the ``bulk
limit'' regime.

For a fixed time moment our model provides a random point
configuration on the one-dimensional lattice --- the so-called Hahn
orthogonal polynomial ensemble. The dynamical model can be described
as a chain of such ensembles with varying parameters, thus it can be
called a dynamical Hahn ensemble. The dynamical correlation kernel
is the so-called extended Hahn kernel; it has been previously
obtained by Johansson \cite{Johansson} and Johansson--Nordenstam
\cite{JohN}. The Eynard-Metha theorem \cite{EM} provides the basis
for computations of the kernel.

The main result of our paper is computation of the asymptotics of
the extended Hahn kernel in the bulk limit regime. The answer is
given by the translation-invariant kernel $K(x,s;y,t)$ defined on
$\Bbb Z^2\times\Bbb Z^2$, for which a simple integral representation
was derived:
  $$K(x,s;y,t)= \frac{1}{2\pi i}\oint_{e^{-i\phi}}^{e^{i\phi}}
  \left(1+cw\right)^{t-s}w^{x-y-1}dw. $$
 Here the integration is to be performed over the right side of the unit circle when  $s\ge t$ and over the left side otherwise, parameters $\phi$ and $c$  depend on the location in the bulk (see Theorem 1).

The static version of the kernel ($s=t$) is the well-known discrete
sine kernel on $\Bbb Z\times\Bbb Z$ \cite{BOkOl}. The obtained
dynamical extension of the sine kernel is among the kernels
constructed in \cite{Bor}. The kernel $K(x,s;y,t)$ is connected with
the kernel obtained by Okounkov and Reshetihin in \cite{Ok2} by a
simple transformation (duality of particles and holes on the
lattice).

In the paper by Kenyon \cite{Ken} local correlations in the model
similar to the one considered by us in the similar limit regimes
were studied. Using different techniques it was shown that the limit
distribution is one of the so-called ergodic gibbs measures
$\mu_{\nu}$ on the infinite hexagonal graph. The latter measures
were classified by Sheffield in his thesis \cite{She}, and in the
article \cite{KOS} correlation kernels for the measures $\mu_{\nu}$
were computed. Kenyon points out that his results should hold in the
general case but in his paper \cite{Ken}, full proof was given under
the assumption that the so-called limit shape has no facets. In our
case the limit shape does have facets, and it seems that our result
does not follow from \cite{Ken}.

In the present article we do not employ the techniques of the papers
by Kenyon, Okounkov and Sheffield and obtain Theorem 1 using a new
method.

Subtle analytical results about asymptotic properties of discrete
orthogonal polynomial ensembles can be found in literature (see
\cite{Baik-Kriecherbauer-McLauphlin-Miller}, \cite{Johansson}), but
they deal with static rather than dynamical ensembles. The usage of
the well-known technique of investigating asymptotics via a suitable
integral representation of the kernel in the case of the extended
Hahn kernel fails too.

We solve the problem using the method first proposed in  \cite{BO2}.
Although in \cite{BO2} again only static models were investigated,
it turned out that the method can be adapted to the dynamical model
of our interest. We use some relations between Hahn polynomials with
varying parameters for the computation of the pre-limit kernel
(similar approach was used in \cite{BO1}).

The paper is organized as follows: In the first section, we define a
stochastic process under study and compute its simplest
characteristics. In Section 2, the determinantal form of the process
is proved and its dynamical correlation kernel is expressed through
the Hahn orthogonal polynomials. In Section 3, the main theorem is
formulated and proved. Finally, in the fourth section we analyze
obtained results and show their connection with theorems proved in
\cite{Cohn-Larsen-Propp} and \cite{Ok2}.

The author would like to thank A.~Borodin for proposing the problem,
advices and hinting the required form of relations for the Hahn
polynomials. The author is grateful to G.~Olshanski for many
fruitful discussions and invaluable help in writing this article.

\section{ Basic definitions and properties }
\subsection{Base model}
Consider the integer lattice $\mathbb{Z}^2$ on the plane $(t,x)$ and
two sets of points on it: $X_i=(0,i-1)$ and $Y_i=(T,S+i-1)$, where
$i$ ranges from $1$ to some fixed $N$, while $T$ and $S$ are some
fixed parameters ($T\ge S$). We are going to study collections of
$N$ non-intersecting paths passing through the lattice nodes, such
that the $i$-th path connects $X_i$ and $Y_i$. Every path is a
polygonal line each link of which is a segment with left end-point
$(t,x)\in\mathbb{Z}^2$ and right end-point $(t+1,x)$ or $(t+1,x+1)$.
In other words, the paths consist of horizontal segments and
segments inclined at an angle of 45 degrees, thereby while moving
along the path the $t$--coordinate increases and the $x$--coordinate
doesn't decrease. We will often call the $t$--coordinate ``time''.

Introduce the probability space of all such families of paths and
the probability measure on it considering all families equiprobable.

Now let us consider an arbitrary time moment $t$ and a set of points
$Z_1,\dots,Z_N$ with coordinates $(t,z_i)$ respectively (we assume
that $z_i$ are arranged in the ascending order). The probability of
the event that our collection of paths passes through the points
$Z_i$ is well defined. This probability induces a probability
measure $P_t(z_1,\dots,z_N)$ on the $N$--point sets of positive
integers in a natural way. Next, the collections of paths under
consideration define a stochastic process with discrete time
$t=1,\dots,T$, the state space consisting of $N$--point sets of
positive integers, and one-dimensional distributions $P_t$. Let us
denote this process by $H_t$. Our aim is  studying properties of
this process.

The above process is a Markov process: this is readily seen from the
very definition of  the Markov property (the past and the future are
independent  given the present).

It is noteworthy that all paths ( and consequently all
configurations of the process $H_t$) are contained inside a hexagon
whose two sides are parallel to the $t$--axis, other two ones are
parallel to the $x$--axis and the remaining two sides are inclined
at an angle of 45 degrees. The corresponding illustration is given
in figure $1$

\begin{center}
 {\scalebox{0.7} {\includegraphics{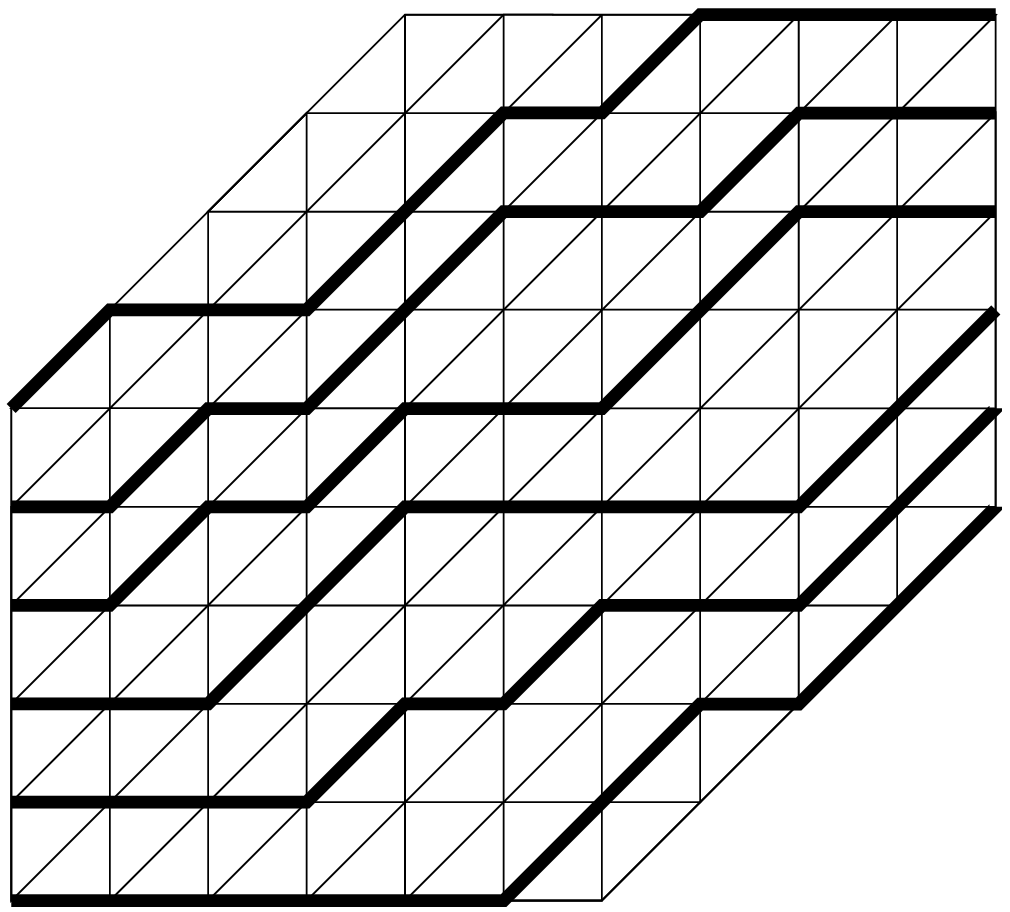}}}

 Figure 1. Collection of 6 non-intersecting paths.
\end{center}

\subsection{Elementary numerical characteristics}

We start with the explicit computation of the quantity
$P_t(z_1,\dots ,z_n)$

\begin{proposition} The number of collections of the non-intersecting paths
connecting points $(t_1,a_i)$ and $(t_2,b_i)$, where $i$ varies from $1$ to
$N$, is equal to $\det_{i,j=1,\dots ,N}\left[{t_2-t_1\choose b_i-a_j}\right]$.
\end{proposition}

\begin{proof}
 This fact is  well known. The trick which is used in the proof is described in  \cite{GV}.
\end{proof}

It follows from Proposition 1 and definitions that
$$ P_t(z_1,\dots ,z_N)=\frac{ \det_{i,j=1,\dots ,N}\left[{t\choose
z_i+1-j}\right]\cdot \det_{i,j=1,\dots ,N}\left[{T-t\choose
S+i-1-z_j}\right]}{ \det_{i,j=1,\dots ,N}\left[{T\choose
S+i-j}\right]}.$$

Next we use Theorem 26 from \cite{Kr} for the computation of the
determinants. After some simplifications we obtain (in what follows
we use the standard notation $(a)_i=a(a+1)\dots (a+i-1)$ ) :

\begin{multline*}
 \prod\limits_{1\leq i<j\leq N} (z_i-z_j)^2
 \prod\limits_{i=1}^N\frac{1}
    {z_i!(t-z_i+N-1)!(S-z_i+N-1)!(T-t-S+z_i)!}\cdot\\
 \cdot\prod\limits_{i=1}^N\frac{(t+1)_{i-1}(T-t+1)_{i-1}(S-i+N)!(T-S+i-1)!}
    {(T+1)_{i-1} (i-1)!}
  \cdot \left(\frac{t!(T-t)!}{T!}\right)^N.
\end{multline*}

After further transformations the probability can be written in the
form:
$$
 \frac{1}{Z}\cdot\prod\limits_{1\leq i<j \leq N} (z_i'-z_j')^2
 \prod\limits_{i=1}^N\left(\frac{(\alpha+1)_{z_i'}(\beta+1)_{M-z_i'}}{z_i'!(M-z_i')!}\right).
$$
Here $Z$ is some normalization constant depending on $T$, $N$, $S$
and $t$, while coordinate $z_i'$ is related to $z_i$  simply through
the shift by a quantity  depending on the remaining parameters . All
explicit formulas connecting different parameters in different
representations of the weight function will be given in the next
section.

Thus, $$
 P_t(z_1,\dots ,z_N)=\frac{1}{Z}\prod\limits_{1\leq i<j \leq N} (z_i'-z_j')^2
 \prod\limits_{i=1}^Nw_{\alpha,\beta,M}(z_i'),$$
where
 $w_{\alpha,\beta,M}(x)=
 \frac{(\alpha+1)_{x}(\beta+1)_{M-x}}{x!(M-x)!}$, $x=0,\dots , M$, is the weight
of \emph{Hahn orthogonal polynomials } (for the definition and
properties of the Hahn polynomials see \cite{KC}) while parameters
$\alpha$, $\beta$, $M$ depend on $T$, $N$, $S$ and $t$; as before,
$z_i'$ is related to $z_i$ by a shift of the origin. In what follows
$w_t(x)$ denotes the weight function at the time moment $t$.

\section{ Determinantal representation. The correlation kernel.}
\subsection{One-dimensional distributions.}
\begin{proposition} Consider an arbitrary non-negative weight $w(x)$ on the set of non-negative integers. Assume that
$p_0, p_1,\dots$ are corresponding orthogonal polynomials. Further
consider the probability measure on the N-tuples of integers
$$P_t(z_1,\dots ,z_N)=
 \frac{1}{Z}\prod\limits_{1\leq i<j \leq N} {(z_i-z_j)^2}
 \prod\limits_{i=1}^Nw(z_i).$$
Then the probability that the random $N$-tuple contains given
distinct points $x_1<x_2<\dots <x_k$ is equal to  $$
\det_{i,j=1,\dots ,k}\left[K(x_i,x_j)\right]$$ where the kernel
 $K$ is computed using formula
 $$K(x,y)=\sqrt{w(x)w(y)}\sum_{n=0}^{N-1}\frac{p_n(x)p_n(y)}{(p_n,p_n)}.$$
\end{proposition}
Proof of this theorem can be found in  \cite{M}

Such {a probability} distribution on the  $N$-tuples of integers is
called \emph{the orthogonal polynomial ensemble}. Observe that
one-dimensional distributions in our problem have precisely the same
type as those in the last proposition.

Now we return to the process $H_t$ under consideration. We are going
to describe explicitly its one-dimensional distributions and the
involved orthogonal polynomials and weight functions.

In our model the weight function can be written in the form
\begin{multline*} w_t(x)=\frac{c}
    {z_i!(t-x+N-1)!(S-x+N-1)!(T-t-S+x)!}\\
   =\frac{(\alpha+1)_{x'}(\beta+1)_{M-x'}}{x'!(M-x')!}\end{multline*}
(here $c$ is some constant that does not depend on $x$ while $x'$
differs from $x$ by some shift). Note that the support of the weight
function varies in time, i.e.,  when $t$ increases by one either the
support doesn't change or one point is being added to it from above
or one point is being removed from below or both actions happen
simultaneously (see Figure 1). Denote the support of $w_t$ by
${\mathfrak X_t}$.

Denote by $Q_k(x',\alpha,\beta,M)$ the Hahn orthogonal polynomial of
degree $k$ with parameters $\alpha, \beta, M$. According to
\cite{KC} these polynomials are defined when $x'=0,\dots ,M$,
$k=0,\dots ,M$, and they are orthogonal with respect to the weight
$$\frac{(\alpha+1)_{x'}(\beta+1)_{M-x'}}{x'!(M-x')!}.$$

The Hahn polynomials can be expressed through the hypergeometric
series:
$$
Q_k(x',\alpha,\beta,M)=
\mathstrut_3F_2{{-k,-x',k+\alpha+\beta+1}\choose {-M,\alpha+1}}(1)
$$

Denote by  $H_k^t(x)$ the Hahn polynomial $Q_k$ corresponding to the
time moment $t$ and shifted in such a way that its domain of
definition coincides with ${\mathfrak X_t}$.

There are 4 cases depending on the relations between $t$ and other
parameters. In each case $\alpha$, $\beta$, $M$ and $x'$ are
expressed in terms of the parameters of our model differently:
$$\begin{array}{l}
  (I)\quad t<S+1,\quad t<T-S+1:\\
  M=t+N-1,\quad
  \alpha=-S-N,\quad
  \beta=S-T-N,\quad
  x'=x,\\
  H_t^k(x)=Q_k(x',\alpha,\beta,M)=
   \mathstrut_3F_2{{-k,-x',k+\alpha+\beta+1}\choose {-M,\alpha+1}}(1)=
   \mathstrut_3F_2{{-k,-x,k-2N-T+1}\choose {-S-N+1,-t-N+1}}(1),\\
   0\le x\le M=t+N-1,\quad
   0\le k\le M=t+N-1;
  \\
  \\
  (II)\quad S-1<t<T-S+1:\\
  M=S+N-1,\quad
  \alpha=-t-N,\quad
  \beta=t-N-T,\quad
  x'=x,\\
  H_t^k(x)=Q_k(x',\alpha,\beta,M)=
   \mathstrut_3F_2{{-k,-x',k+\alpha+\beta+1}\choose {-M,\alpha+1}}(1)=
   \mathstrut_3F_2{{-k,-x,k-2N-T+1}\choose {-S-N+1,-t-N+1}}(1),\\
   0\le x\le M=s+N-1,\quad
   0\le k\le M=s+N-1;
  \\
  \\
  (III)\quad T-S-1<t<S+1:\\
  M=t+N-S-1,\quad
  \alpha=-T+t-N,\quad
  \beta=-t-N,\quad
  x'=T-t-s+x,\\
  H_t^k(x)=Q_k(x',\alpha,\beta,M)=
   \mathstrut_3F_2{{-k,-x+t+s-T,k-2N-T+1}\choose {-T+t-N+1,-T+S-N+1}}(1),\\
   -T+t+S\le x'\le t+N-1,\quad
   0\le k\le T+N-S-1;
  \\
  \\
  (IV)\quad t>T-S-1,\quad t>S-1:\\
  M=T-t+N-1,\quad
  \alpha=-T-N+S,\quad
  \beta=-S-N,\quad
  x'=T-t-s+x,\\
  H_t^k(x)=Q_k(x',\alpha,\beta,M)=
   \mathstrut_3F_2{{-k,-x+t+s-T,k-2N-T+1}\choose {-T+t-N+1,-T+S-N+1}}(1),\\
   -T+t+S\le x'\le S+N-1,\quad
   0\le k\le T+N-t-1.

\end{array}$$

 It is noteworthy that the expressions of polynomials  $H_k^t$ in terms of hypergeometric series coincide in cases
 (I) and (II). The same is true for cases (III) and (IV).

 Let us introduce one more notation:
 $$f_n^t(x)=\frac{H_n^t(x)\sqrt{w_t(x)}}{\sqrt{(H_n^t,H_n^t)}}.$$

 Here $(H_n^t,H_n^t)$ is the squared norm of the polynomial  $H_n^t$ with respect to the weight function $w_t(x)$.
 According to \cite{KC}
 \begin{multline*}
 (H_n^t,H_n^t)=(Q_k(x',\alpha,\beta,M),Q_k(x',\alpha,\beta,M))\\=
  \frac{(-1)^k (k+\alpha+\beta+1)_{N+1}(\beta+1)_k k!}
   {(2k+\alpha+\beta+1)(\alpha+1)_k (-N)_k N!}.\end{multline*}

 The functions  $f_n^t(x)$ form an orthonormal basis in the space
 $l_2({\mathfrak X_t})$.

 Now the kernel $K$ from Proposition 2 can be rewritten as:
 $$K(x,y)=\sum_{n=0}^{N-1}f_n^t(x)f_n^t(y).$$

\subsection{Dynamical determinantal property}
Let $X(t)$ be a stochastic process taking values in the space of all
subsets of some countable set. Let us say that $X(t)$ is a
\emph{determinantal} process if there is a function $K(x,s;y,t)$
such that for any $n$--tuple of distinct pairs
$(x_1,t_1),\dots,(x_n,t_n)$ in the space--time
$${\rm Prob}\{x_1\in X(t_1),x_2\in X(t_2),\dots,x_n\in X(t_n)\}=
  \det_{i,j=1,\dots ,n}\left[K(x_i,t_i;x_j,t_j)\right].$$

We will call the function  $K$  \emph{the dynamical correlation
kernel}. The above probabilities will be called the dynamical
correlation functions.

To prove that $H_t$ is determinantal we need the following abstract
proposition.

\begin{proposition} Assume that for every time moment $t$ we are given
 an orthonormal system $\{f^t_n\}$ in linear space
$l_2(\{0,1, \dots,L\})$ (it may form the basis of the whole space or only of
$l_2(L_t)$, where $L_t \subset \{0,..,\dots L\}$) and a set of numbers
 $c_0^t,c_1^t, \dots$. Denote
$$
v_{t,t+1}(x,y)=\sum_{n\ge 0}c_n^t f^t_n(x)f^{t+1}_n(y).
$$

 Assume also that we are given a discrete time  Markov process
$X_t$ taking values in $N$--tuples of elements of the set  $\{0,1,
\dots,L\}$, with one--dimensional distributions
$$(\det_{i,j=1,\dots ,N}\left[f_{i-1}^t(x_j)\right])^2$$
 and transition probabilities $$\frac{\det\left[v_{t,t+1}(x_i,y_j)\right]
 \det\left[f_{i-1}^{t+1}(y_j)\right]} {\det\left[f_{i-1}^t(x_j)\right]
  \prod\limits_{n=0}^{N-1}c_n^t}. $$

  Then
\begin{multline*}
  p_n(x_1,k_1;\dots ;x_n,k_n):=
 {\rm Prob}\{x_1\in X_{k_1},\dots, x_n\in X_{k_n}\}\\
 =\det_{i,j=1,\dots ,n}\left[K(x_i,k_i;x_j,k_j)\right],
\end{multline*}
where
\begin{gather*}
K(x,k;y,l)=\sum_{i=0}^{N-1}\frac{1}{c_i^{l,k}}f_i^k(x)f_i^l(y)
 ,\, k\ge l; \\ K(x,k;y,l)=-\sum_{i\ge N}c_i^{k,l} f_i^k(x)f_i^l(y),
 \,k<l; \\ c_i^{k,k}=1,\, c_i^{k,l}=c_i^k\cdot c_i^{k+1}\cdot\dots\cdot c_i^{l-1}
 \end{gather*}
 \end{proposition}

 This proposition is just a slight extension of the theorem proved by Einard and Metha
 in the paper
 \cite{EM}. In their article this proposition was proved under the assumption
 that the system  $\{f^t_n\}$ forms a basis of the whole space.
 But this assumption is not essential as we always can extend an orthonormal system to a
 basis of the whole space  and then set the coefficients
 $c^t_i$ corresponding to the extra vectors equal to zeros.
 This operation doesn't change neither transition probabilities nor one-dimensional distributions
 of the process.

 Another similar formulation of the proposition was used in  \cite{BO1}.

 It is noteworthy that the dynamical correlation kernel is not canonically defined. Only dynamical
 correlation functions are invariants of the process and they don't change when we transform
 the correlation kernel in the following way:
 $$ K^*(x,s;y,t)=K(x,s;y,t) \cdot \frac{F(x,s)}{F(y,t)},$$
 here $F$ is arbitrary function. It is evident that such transformation doesn't change the values of determinants.

 \begin{lemma} One-step transition probabilities of the process  $H_t$ can be computed by
formula
\begin{multline*}
P_{t,t+1}(x_1,\dots ,x_N;y_1,\dots ,y_N):=\\=
{\rm
Prob}\{H_{t+1}=(y_1,\dots ,y_N)\mid H_{t}=(x_1,\dots ,x_N)\}=\\=
 \frac{\prod\limits_{i<j}(y_j-y_i)
   \prod\limits_{i:\,y_i=x_i+1}(N+S-x_i-1)\prod\limits_{i:\,y_i=x_i}(x_i+T-t-S)}
 {(T-t)_{N}\cdot\prod\limits_{i<j}(x_j-x_i)},
\end{multline*}
provided that each difference $y_i-x_i$ is equal to zero or to one,
otherwise the probabilities are equal to zero.
\end{lemma}

This formula can be checked by direct computation starting from the
definition and using Proposition 1 together with the determinantal
formula from \cite{Kr} used in Section 1.2.

 \begin{proposition}
  The process $H_t$ meets the assumptions of Proposition 3 and thus it is determinantal.
 \end{proposition}

  Here the orthogonal system from the formulation of Proposition 3 is the system
 of functions  $f_n^t$ defined in Section 2.1 while
  coefficients $c^t_i$ are given by the following formula
  $$ c^t_i=\sqrt{\left(1-\frac{i}{t+N}\right)\left(1-\frac{i}{T+N-t-1}\right)}$$
 \begin{proof}
  We will need 3 lemmas.
 \begin{lemma}
  The following relations hold for the Hahn orthogonal polynomials:
  \begin{multline*}
   (*) \quad xQ_k(x-1;\alpha,\beta,M-1)+(N-x)Q_k(x;\alpha,\beta,M-1)
    =MQ_k(x;\alpha,\beta,M),
  \\
   (**) \quad xQ_k(x-1;\alpha +1,\beta -1,M)+(-x-\alpha-1)Q_k(x;\alpha+1,\beta-1,M)
    =\\=-(\alpha+1)Q_k(x;\alpha,\beta,M).
  \end{multline*}
  They are equivalent to the following relation for the function $_3F_2$:
  $$
  x\cdot _3F_2{{-k,a,-x+1} \choose {b,c+1}}(1)-(c+x)\cdot_3F_2{{-k,a,-x}\choose{b,c+1}}(1)
  =-c\cdot _3F_2{{-k,a,-x}\choose{b,c}}(1).
  $$
 \end{lemma}
 \begin{proof}
 The following formula holds (see \cite{KC}):
 \begin{multline*}
  Q_k(x,\alpha,\beta,M)=
 \mathstrut_3F_2{{-k,-x,k+\alpha+\beta+1}\choose {-M,\alpha+1}}(1)\\=
  \sum_{i=0}^k\frac{(-k)_i(-x)_i(k+\alpha+\beta+1)_i}{(-M)_i
  (\alpha+1)_i i!}.
 \end{multline*}
 It implies
\begin{gather*}
   xQ_k(x-1;\alpha,\beta,M-1)+(N-x)Q_k(x;\alpha,\beta,M-1)
    -NQ_k(x;\alpha,\beta,M)=\\
 =\sum_{i=0}^k\frac{(-k)_i(k+\alpha+\beta+1)_i}{(\alpha+1)_ii!}
    \left[ x\frac{(-x+1)_i}{(-M+1)_i}+(M-x)\frac{(-x)_i}{(-M+1)_i}-
           M\frac{(-x)_i}{(-M)_i}\right]=\\
=\sum_{i=0}^k\frac{(-k)_i(-x)_i(k+\alpha+\beta+1)_i}{(-M+1)_i(\alpha+1)_ii!}
    \left[(x-k)+(M-x)-(M-k)\right]=0.
  \end{gather*}
 The second relation is readily seen to be equivalent to the one already proved.
 \end{proof}

 \begin{lemma}
 \begin{multline*}
  \sum_{k=0}^M\frac{(2k+\alpha+\beta+1)(\alpha+1)_{k}(-M)_{k}M!}
    {(-1)^k(k+\alpha+\beta+1)_{M+1}(\beta+1)_{k}k!}Q_k(x,\alpha,\beta,M)
        Q_k(y,\alpha,\beta,M)=\\
      =\frac{\delta_{x,y}}{
       \dfrac{(\alpha+1)_{x}(\beta+1)_{M-x}}{x!(M-x)!}}
 \end{multline*}
 \end{lemma}
 This formula is called ``dual orthogonality relation'' and can be found in
\cite{KC}.

 \begin{lemma}
 \begin{multline*}
 \sqrt{\frac{w_{t+1}(y)}{w_t(x)}}\sum_{k\ge 0}c_k^t
        f_k^t(x)f_k^{t+1}(y)=\\=
  \frac{S+N-1-x}{\sqrt{(t+N)(T+N-t-1)}}\delta_{x+1}^y+
  \frac{T-t-S+x}{\sqrt{(t+N)(T+N-t-1)}}\delta_x^y.
 \end{multline*}
 Here coefficients $c_i^t$ and functions $f_k^t(x)$ were defined in the formulation of Proposition 5, $w_t(x)$ is the
 weight function of the Hahn orthogonal polynomials.

 \end{lemma}
 \begin{proof}
  To prove this it suffices to recall the definition of
  $f_k^t(x)$, substitute the value of the weight function and
  all parameters, then in cases (I) and (II) express $f_k^{t+1}(y)$ in terms of $f_k^t(y)$
  and $f_k^t(y-1)$ using the relation  (*), and finally use the last lemma.
  In the cases (III) and (IV) one should express  $f_k^t(x)$ in terms of
  $f_k^{t+1}(x)$ and $f_k^{(t+1)}(x+1)$ using (**) and again use
  the last lemma.
 \end{proof}

 Let us return to the proof of the proposition.

 The fact that the one-dimensional distributions can be expressed in the form that is required for
 the application  of Proposition 5 follows from Proposition 2, where one takes
 $k=N$  (actually, in this case Proposition 2 is quite obvious) and the following observation:
 \begin{multline*}(\det_{i,j=1,\dots ,N}\left[f_{i-1}^t(x_j)\right])^2=
    \det_{i,j=1,\dots ,N}\left[f_{i-1}^t(x_j)\right]\cdot
    \det_{i,j=1,\dots ,N}\left[f^t_{j-1}(x_i)\right]\\=
    \det\left(\left[f^t_{i-1}(x_j)\right]\cdot \left[f^t_{j-1}(x_i)\right]\right)=
    \det_{i,j=1,\dots ,N}\left[K(x_j,x_i)\right]\end{multline*}
 (here $K(x,y)$ is the kernel defined in Propostition 2).

 It remains to show that the transition probabilities also have the required form, i.e:
  $$P_{t,t+1}(x_1,\dots ,x_N;y_1,\dots ,y_N)=\frac{\det\left[v_{t,t+1}(x_i,y_j)\right]
 \det\left[f_{i-1}^{t+1}(y_j)\right]} {\det\left[f_{i-1}^t(x_j)\right]
  \prod\limits_{n=0}^{N-1}c^t_n} $$
 (all matrices under the sign of determinant have order $N\times N$ and we will not indicate
 this in the sequel).
 Here
  $$v_{t,t+1}(x,y)=\sum_{k\ge 0}c^t_k f^t_k(x)f^{t+1}_k(y).$$
 By virtue of already proved facts we obtain:
 \begin{multline*}
   v_{t,t+1}(x,y)=\\=\sqrt{\frac{w_t(x)}{w_{t+1}(y)}}
  \left[\frac{S+N-1-x}{\sqrt{(t+N)(T+N-t-1)}}\delta_{x+1}^y+
  \frac{T-t-S+x}{\sqrt{(t+N)(T+N-t-1)}}\delta_x^y\right]\end{multline*},
 $$\det\left[f_{i-1}^{t+1}(y_j)\right]=\sqrt{P_{t+1}(y_1,\dots,y_N)},$$
 $$\det\left[f_{i-1}^{t}(x_j)\right]=\sqrt{P_t(x_1,\dots,x_N)}.$$
 Substituting the expressions of the weight function and of the probability of
 the $N$--tuple we get:
 $$
  \frac{\det\left[v_{t,t+1}(x_i,y_j)\right]
 \det\left[f_{i-1}^{t+1}(y_j)\right]} {\det\left[f_{i-1}^t(x_j)\right]
  \prod\limits_{i=0}^{N-1}c_i^t}
 =$$ $$=
 \frac{ \prod\limits_{i=1}^N
          \sqrt{\frac{w_{t+1}(y_i)}{w_t(x_i)}
           \frac{(t+i)(T-t)(t+1)^N}{(t+1)(T-t+i-1)(T-t)^N}
          }
        \det\left[v_{t,t+1}(x_i,y_j)\right]
  }
  {\prod\limits_{i=0}^{N-1}c_i^t}
 \prod\limits_{i<j}\frac{y_j-y_i}{x_j-x_i}
 =$$ $$=
 \frac{\det\left[(S+N-x-1)\delta_{x_i+1}^{y_j}+(T-t-S+x)\delta_{x_i}^{y_j}\right]}
          {(T-t)_{N}}\prod\limits_{i<j}\frac{y_j-y_i}{x_j-x_i}.
 \eqno (***)
 $$

 Let us assume that $x_i$ and $y_i$ are arranged in the increasing order.

 There are two cases: either every difference $y_i- x_i$
 is equal 0 or 1 or this is not true.

 In the latter case the matrix under the sign of determinant has a zero row or column.
 Consequently, the determinant is equal to zero. On the other hand, the transition
 probability is equal to zero, too.

 In the former case the matrix under the sign of determinant is block-diagonal.
 Moreover, its zeros are arranged in such a way
 that every block is either an upper triangular or lower triangular matrix.
 We may conclude that the determinant of the matrix is equal to the product of
 its  diagonal elements, and
 the value of  $(***)$ is equal to the transition probability of the process
  $H_t$ computed in Lemma 4.

 Hence the process meets the conditions of Proposition 3 and we conclude that the process is determinantal.
 \end{proof}

\section{Passing to the limit in the correlation kernels: the operator method}
\subsection{Statement of the result}
 We are interested in the limit distribution (limit correlation kernel) in the following limit regime:

 Let us fix numbers  $\tilde S$, $\tilde T$, $\tilde N$, $\tilde t$, $\tilde x$.
 Assume that auxiliary parameter  $\rho \to \infty$ and
 $$
 \begin{array}{llll}
   S=\rho \tilde S + o(\rho),&
   T=\rho \tilde T + o(\rho),&
   N=\rho \tilde N + o(\rho),&
   t=\rho \tilde t + \hat t,\\
   s=\rho \tilde t + \hat s,&
   x=\rho \tilde x + \hat x,&
   y=\rho \tilde x + \hat y.
  \end{array}
  $$

  This limit has a simple geometrical interpretation:
  we fix the proportions between the parameters
   $\tilde T$, $\tilde S$, $\tilde N$ of the hexagon containing the paths
  and a point  $(\tilde t, \tilde x)$ inside it. Further we
  enlarge the hexagon by means of homothety with center at
  $(\tilde t,\tilde x)$ and coefficient
  equal to  $\rho$. Thus, the number of the paths increases, and we are
  focusing on the local picture near the point $(\tilde t,\tilde x)$.

  Recall the formula for the pre-limit kernel:
  $$
  \begin{array}{ll}
   K(x,s;y,t)=&
   -\sum\limits_{i\ge N}\left(\prod\limits_{j=s}^{t-1}c^j_i\right)f_i^s(x)f_i^t(y)
   ,\quad s<t\\
  &\sum\limits_{i=0}^{N-1}\left(\prod\limits_{j=t}^{s-1}\frac{1}{c^j_i}\right)f_i^s(x)f_i^t(y)
   ,\quad s\ge t.
  \end{array}
  $$
  The functions $f_i^s(x)$ and the coefficients $c_i^j$ were defined in Section 2.

  \begin{theorem}
  Let all the parameters of the original model
 ($S,T,N,t,s,x,y$) tend to infinity  in such a way that
  $$
   \begin{array}{llll}
   S=\rho \tilde S + o(\rho),&
   T=\rho \tilde T + o(\rho),&
   N=\rho \tilde N + o(\rho),&
   t=\rho \tilde t + \hat t,\\
   s=\rho \tilde t + \hat s,&
   x=\rho \tilde x + \hat x,&
   y=\rho \tilde x + \hat y,
  \end{array}
  $$
  where the auxiliary parameter $\rho \to \infty$.
  Then the limit point process on ${\mathbb Z^2}$ exists.
  It is determinantal and has translation-invariant kernel
  $$K(\hat x,\hat s;\hat y,\hat t)= \frac{1}{2\pi i}\oint_{e^{-i\phi}}^{e^{i\phi}}
  \left(1+cw\right)^{\hat t-\hat s}w^{\hat x-\hat y-1}dw. $$
  Here the integration is to be performed over the right side of the unit
  circle when
   $s\ge t$ and over the left side otherwise,

 $$ c= \sqrt{\frac{\tilde x(\tilde S+\tilde N-\tilde x)}{(\tilde T-\tilde t-\tilde S+\tilde x)(\tilde t+\tilde N-\tilde
x)}},
 $$
 and angle $\phi$ is given by formula:
 $$\phi=
   \arccos\frac{-\tilde N(\tilde N+\tilde T)+(-\tilde x+\tilde S+\tilde N)
   (\tilde t+\tilde N-\tilde x)+
           \tilde x(\tilde T+\tilde x-\tilde S-\tilde t)}
   {2\sqrt{\tilde x(-\tilde x+\tilde S+\tilde N)
      (\tilde t+\tilde N-\tilde x)(\tilde x+\tilde T-\tilde S-\tilde t)}}.
 $$
 If the expression under the arc cosine sign is greater than 1, then we set $\phi=0$. If the expression
 is less than $-1$, then $\phi=\pi$.
 \end{theorem}

 Existence of the process given by the limit correlation kernels follows easily from the convergence of
 these functions. The proof of this fact can be found for instance
 in \cite{Bor} (Lemma 4.1).

 The proof of the convergence of the correlation kernels will be made in two steps.

 \subsection{Statical case}
 First assume that $\hat s-\hat t=0$, i.e. we are investigating only the ``statical'' distribution
 in some fixed time moment.

 We use the method first proposed by A.~Borodin and G.~Olshanski in
 the paper \cite{BO2}.

 We extend the functions  $f_i^t(x)$ belonging to the space
 $l_2({\mathfrak X_t})$ to the functions defined on  $\mathbb Z$ by setting them equal to 0
 outside $\mathfrak X_t$.  Now we may regard all the functions $f_i^t$ as belonging to
 one and the same space  $l_2({\mathbb Z})$.

 We have
 $$K_t(x,y)=\sum\limits_{i=0}^{N-1}f_i^t(x)f_i^t(y).$$
 Observe that $K_t(x,y)$ is the matrix element of the projection operator
 onto the subspace spanned by the first  $N$ orthogonal functions $f_i^t$, $i=0,\dots,N-1$,
 in the space $l_2({\mathbb Z})$.

 It turns out that finding the limit of the operator is easier than
 computing the limit of the matrix elements. Note that functions
  $f_n^t(x)$ are eigenvectors of some difference operator (it will be explicitly given below).
  The projection operator can be regarded as the spectral projection on the segment containing
 the first $N$ eigenvalues of the difference operator under investigation.
 Now, to find the limit of the spectral projection operators we will take the limit of the difference
 operators. Note that both the difference operator and the spectral
 segment are varying simultaneously.

 To justify the limit transition we use some theorems from functional analysis.

 Consider the set  $l_2^0({\mathbb Z})$ of the finite vectors from
 $l_2({\mathbb Z})$ (i.e. algebraic span of the basis elements $\delta_x$)
 as a common essential domain of all considered difference
 operators.
 It will be clear from the following that the difference operators strongly converge
 on this domain.
 It follows that operators converge in the strong resolvent sense  (see \cite{RS}, Theorem VIII.25).
 The last fact, continuity of the spectrum of the limit operator and Theorem
 VIII.24 from \cite{RS} imply that the spectral projections associated with the difference
 operators strongly converge on the set of finite vectors to the limit
 spectral projection associated with the limit difference operator.

 Now we will carry the explicit computation through.

 The following difference relation holds for the Hahn polynomials (see \cite{KC}):
 \begin{multline*}
   k(k+\alpha+\beta+1)Q_k(x';\alpha,\beta,M)=
 (x'+\alpha+1)(M-x)Q_k(x'-1;\alpha,\beta,M)+\\+
 ((x'+\alpha+1)(M-x)+x'(x'-\beta-M-1))Q_k(x';\alpha,\beta,M)+\\  +x'(x'-\beta-M-1)Q_k(x'+1;\alpha,\beta,M)
 .\end{multline*}

 Let us rewrite this relation in the terms of orthonormal functions
$f^t_k(x')$.
 It must be emphasized that here the variable  $x'$ is used instead of the variable  $x$ (It was shown
 in Section 2.1 that in cases (I) and (II) $x=x'$, while in cases (III) and (IV)
  $x \neq x'$).
 $$\begin{array}{l}
  \frac{k(k+\alpha+\beta+1)}{\rho^2}f^t_k(x')=
 \frac{\sqrt{(x'+1)(M-x')(\beta+M-x')(\alpha+x')}}{\rho^2}f^t_k(x'+1)-\\
  -\frac{(x'+\alpha+1)(x'-M)+'x(x'-\beta-M-1)}{\rho^2}f^t_k(x')+
  \frac{\sqrt{x'(M-x'+1)(\beta+M-x')(\alpha+x')}}{\rho^2}f^t_k(x'-1).
 \end{array}$$
 We divided the relations by $\rho^2$ to ensure the existence of the limits that will be considered below.

 Denote by $H$ the difference operator given by the right side of the last equation for
 those $x$ for which the Hahn polynomials are defined and, for instance, as identity operator
 for other $x$. The matrix of the operator $H$ is given by
 $$H(x,y)= \begin{cases} K_t(x,y),&x,y\in {\mathfrak X_t}\\
      \delta_x^y &\text{for other } x,y.
   \end{cases}
 $$

 It is readily seen that $H$ is a self-adjoint operator.

 We are interested in the projection on the eigenvectors corresponding to the
 part of the spectrum for  $k\in[0, \dots ,N-1]$, i.e. corresponding to the eigenvalues
 $$\frac{k(k+\alpha+\beta+1)}{\rho^2},\quad k\in[0 \dots N-1].$$
 Observe that the full spectrum of the operator $H$ consists of the negative numbers that
 can be obtained by last formula when $k$ ranges from $0$ to $M$ and the value $1$ that occurred because
 of the formal continuation of the operator. In our model
 $\alpha$ and $\beta$ are rather great in magnitude negative numbers and the function
 $k(k+\alpha+\beta+1)$ is negative and monotonic for the permissible values of $k$.
 Consequently, we can replace the projection on the eigenspaces
 associated with the discrete set of points of the spectrum by the projection associated
 with some segment containing them. Thus the desired projection $P$
 is the spectral projection associated with the following segment:
 $$
  \left[\frac{(N-1)(N+\alpha+\beta)}{\rho^2},0\right].
 $$

 Substitute the values of  $\alpha$, $\beta$ and other parameters and express $H$ and the projection $P$ through
 $\rho$ and the parameters
 $\tilde S$, $\tilde T$, $\tilde N$, $\tilde t$, $\tilde x$:
   $$
   \begin{array}{llll}
   S=\rho \tilde S + o(\rho),&
   T=\rho \tilde T + o(\rho),&
   N=\rho \tilde N + o(\rho),\\
   t=\rho \tilde t + \hat t,&
   x=\rho \tilde x + \hat x.
  \end{array}$$

 We obtain the family of the operators  $P_\rho$ in the space
 $l_2({\mathbb Z})$ where $\hat x$ is the variable.
 Our aim is finding the limit
 $$\hat P = \lim_{\rho \to \infty}P_\rho.$$

 For the application of Theorem VII.24 from \cite{RS} it is necessary that
 the family of the difference operators
 $H_\rho$ obtained by substituting the parameters of the limit regime
 converges to some limit $\hat H$ in the sense of the strong
 operator convergence on the space of the finite vectors.
 This convergence follows from the convergence of the
 matrix elements because in our case all operators $H_\rho$
 are given by tridiagonal matrices.

 Taking the limit of the matrix elements we derive that
  $\hat H$ is a symmetric operator which is given by tridiagonal matrix
  that has the number
 \begin{multline*}
  A=\lim_{\rho \to \infty}-\frac{(x'+\alpha+1)(x'-M)+x'(x'-\beta-M-1)}{\rho^2}=\\
  =-(\tilde S+\tilde N-\tilde x)(\tilde t+\tilde N-\tilde x)-\tilde x (\tilde x+\tilde T-\tilde S-\tilde t)
 \end{multline*}
 on the diagonal and the number
 \begin{multline*}
 B=\lim_{\rho \to \infty}
  \frac{\sqrt{(x'+1)(M-x')(\beta+M-x')(\alpha+x')}}{\rho^2}=\\=
 \lim_{\rho \to \infty}
  \frac{\sqrt{(x')(M-x'+1)(\beta+M-x')(\alpha+x')}}{\rho^2}=\\=
 \sqrt{(\tilde S+\tilde N-\tilde x)(\tilde t+\tilde N-\tilde x)\tilde x (\tilde x+\tilde T-\tilde S-\tilde t)}
 \end{multline*}
 above the diagonal and below it.

 It is noteworthy that although the parameters  $\alpha$, $\beta$, $M$ and $x'$)
 were expressed through the characteristics of the model
 in the different ways (we had 4 cases) here these differences
 disappear and the limit difference operator is the same in all
 cases.

 The limit spectral segment can be easily computed also:
 $$\lim_{\rho\to\infty}\left[\frac{(N-1)(N+\alpha+\beta)}{\rho^2},0\right]=
   \left[-\tilde N (\tilde N+\tilde T),0\right].$$
 Denote by  $I_{[a,b]}$ the indicator of the segment $[a,b]$.

 Observe that the following identity holds:
 $$
  \hat P =I_{\left[-\tilde N (\tilde N+\tilde T),0\right]}(\hat H)=
   I_{\left[\frac{-\tilde N (\tilde N+\tilde T)-A}{2B},-\frac{A}{2B} \right]}\left( \frac{\hat H - A}{2B}\right)
 .$$

 $ \frac{\hat H - A}{2B}$ is the operator which is given by the matrix with
 the number $\frac{1}{2}$ above and below the diagonal and zeros in all other entries.

 To compute explicitly  $\hat P$ we make the Fourier transform $l_2({\mathbb Z})\to L_2(S^1)$ where $S^1$
 is the unit circle in ${\mathbb C}$.

 The Fourier transform of  $ \frac{\hat H - A}{2B}$ is the operator of multiplication
  by the function  $\frac{z+\overline z}{2}=\Re z$.
 Observe that the last operator has purely continuous spectrum that fills the segment $[-1,1]$.
 Hence, the spectrum of the operator $\hat H$ is purely continuous too (this condition is necessary
 for applying Theorem  VII.24 from \cite{RS}).

 Next, we substitute $A$ and $B$. The coordinate of the left end-point of the spectral segment
 is equal to
  $$\frac{-\tilde N(\tilde N+\tilde T)+(-\tilde x+\tilde S+\tilde N)(\tilde t+\tilde N-\tilde x)+\tilde X(\tilde T+\tilde x-\tilde S-\tilde t)}
                   {2\sqrt{\tilde x(-\tilde x+\tilde S+\tilde N)(\tilde t+\tilde N-\tilde x)(\tilde x+\tilde T-\tilde S-\tilde t)}},$$
 while the coordinate of the right one is equal to
 $$ \frac{(-\tilde x+\tilde S+\tilde N)(\tilde t+\tilde N-\tilde x)+\tilde x(\tilde T+\tilde x-\tilde S-\tilde t)}
                   {2\sqrt{\tilde x(-\tilde x+\tilde S+\tilde N)(\tilde t+\tilde N-\tilde x)(\tilde x+\tilde T-\tilde S-\tilde t)}}
 .$$

 Note that the value of the last expression is greater than one.
 Therefore, the Fourier transform of the spectral projection
 $\hat P$ becomes the operator of multiplication by the characteristic function
 of the right arc of the unit circle contained between the angles
 $-\phi$ and $\phi$ where
 $$\phi=\arccos\frac{-\tilde N(\tilde N+\tilde T)+(-\tilde x+\tilde S+\tilde N)(\tilde t+\tilde N-\tilde x)+\tilde x(\tilde T+\tilde x-\tilde S-\tilde t)}
                   {2\sqrt{\tilde x(-\tilde x+\tilde S+\tilde N)(\tilde t+\tilde N-\tilde x)(\tilde x+\tilde T-\tilde S-\tilde t)}}
 .$$

 Finally we perform the inverse Fourier transform and find the
  matrix of the operator $\hat P$ which coincides with the desired
  correlation kernel
 $$
  K_{\hat t}(\hat x,\hat y)=K(\hat x,\hat y)=
  \frac{1}{2\pi i}\oint_{e^{-i\phi}}^{e^{i\phi}} w^{\hat y-\hat x-1}dw=
  \frac{1}{2\pi i}\oint_{e^{-i\phi}}^{e^{i\phi}} w^{\hat x-\hat y-1}dw
 $$
 (the integration is to be performed over the right side of the unit circle).

 The last integral can be computed explicitly and it coincides with the discrete sine
 kernel:
 $$
  K(\hat x,\hat y)=\frac{\sin (\phi(x-y))}{\pi (x-y)}.
 $$

 Observe that applying the same method we could have found the limit of the kernel
 $$\breve K_t(x,y)=-\sum\limits_{i\ge N}f_i^t(x)f_i^t(y).$$
 This limit is equal to
 $$
   \breve K_{\hat t}(\hat x,\hat y)=\breve K(\hat x,\hat y)=
  \frac{1}{2\pi i}\oint_{e^{-i\phi}}^{e^{i\phi}} w^{\hat x-\hat y-1}dw,
 $$
 where the integration is to be performed over the left side of the unit circle.
 In the space $L_2(S^1)$ we obtain the operator of multiplication by the characteristic function
  of the left arc of the unit circle.

 \subsection{ Passage to the dynamical kernel}

 Now assume that $s<t$.
 We have
 $$
  K(x,s;y,t)=-\sum_{i\ge N}\left(\prod_{j=s}^{t-1}c^j_i\right)f_i^s(x)f_i^t(y)
 .$$
 Fix $s$ and $t$ and let $\cal Q$ be the operator in $l_2({\mathbb Z})$ which is given
 by matrix
 $$
  {\cal Q}(x,y)=\begin{cases} K(x,s;y,t), &x\in {\mathfrak X_s}, y\in {\mathfrak X_t},\\
                 0 &\text{for other } x,y.
 \end{cases}
 $$
 This operator can be expressed in the following way:
 $${\cal Q}={\cal P}_t \cdot U_{t-1}\cdot U_{t-2}\cdots U_{s}.$$
 Here ${\cal P}_t$ is the operator corresponding to the kernel $\breve K_t(x,y)$ considered in the end of the
 last section, this operator is given by matrix
 $$
  {\cal P}_t(x,y)=\begin{cases} -\sum\limits_{i \ge N}f_i^t(x)f_i^t(y), &x,y\in {\mathfrak X_t},\\
                 0 &\text{for other } x,y,
   \end{cases}
 $$
 and $U_h$ is the operator mapping every function $f_i^h(x)$ into
 $c^h_if^{h+1}_i(x)$ (if in the time moment
 $h+1$ one less polynomial exist, then the function corresponding to the polynomial of the highest power
 is mapped to zero; this definition agrees with the fact that corresponding coefficient $c^h_i$ is equal to zero).
 On the orthogonal complement to the span of the functions $f^h_i$ set $U_h$ equal to zero.
 This operator can be given by the matrix
 $$
   U_h(x,y)=\begin{cases} \sum\limits_{i \ge 0}c_i^hf_i^h(x)f_i^{h+1}(y),
        &x\in{\mathfrak X_h},y\in{\mathfrak X_{h+1}},\\
            0 &\text{for other } x,y.
  \end{cases}
 $$

 Now we express all parameters in the terms
 of characteristics of the limit regime and obtain the family of the operators
 depending on the parameter $\rho$.
 To compute the limit  ${\cal Q}_\rho$ when $\rho \to \infty$
 we take the limit of every factor separately and then use the following argument:
 all factors strongly converge on the dense set of the finite vectors, the norm of each factor
 is not greater than $1$  (For the operators $U_h$ it follows from the fact that
 $0\le c_i^j \le 1$), consequently, operators strongly converge on the whole space
  $l_2({\mathbb Z})$. The operation of multiplication is continuous
  in the strong operator topology if norms of all factors are uniformly bounded.
 This argument implies the strong convergence of the operators  $\cal Q_\rho$.

 Observe that the matrix $U_h$ have already been computed in the proof of the determinantal property of the
 process:
 \begin{multline*}
  U_h(x,y)=v_{h,h+1}(x,y)=\\=\sqrt{\frac{w_h(x)}{w_{h+1}(y)}}
  \left[\frac{S+N-1-x}{\sqrt{(h+N)(T+N-h-1)}}\delta_{x+1}^y+
  \frac{T-h-S+x}{\sqrt{(h+N)(T+N-h-1)}}\delta_x^y\right]=\\=
  \delta_{x+1}^y\sqrt\frac{(S+N-x-1)(x+1)}{(h+N)(T+N-h-1)}+
  \delta_x^y\sqrt\frac{(T-h-S+x)(h+N-x)}{(h+N)(T+N-h-1)}\\
    (x\in {\mathfrak X_h}, y\in {\mathfrak X_{h+1}}).
 \end{multline*}

 Next, we substitute in the last expression all parameters of the limit regime
 $$ \begin{array}{lll}
   S=\rho \tilde S + o(\rho),&
   T=\rho \tilde T + o(\rho),&
   N=\rho \tilde N + o(\rho),\\
   t=\rho \tilde t + \hat t,&
   t=\rho \tilde t + \hat s,&
   x=\rho \tilde x + \hat x
  \end{array}$$
  and obtain the operators $U_{h,\rho}$.
  Taking into account that $s\le h<t$ we send  $\rho$ to infinity.
  We obtain the operator which is given by the matix
  $$\hat U(\hat x,\hat y)=
  \delta_{\hat x+1}^{\hat y}\sqrt\frac{(\tilde S+\tilde N-\tilde x)(\tilde x)}{(\tilde t+\tilde N)(\tilde T+\tilde N-\tilde t)}+
  \delta_{\hat x}^{\hat y} \sqrt\frac{(\tilde T-\tilde t-\tilde S+\tilde x)(\tilde t+\tilde N-\tilde x)}{(\tilde t+\tilde N)(\tilde T+\tilde N-\tilde t)}
  .$$

  Now we make the Fourier transform and multiply all factors. The result is the limit operator
  $\cal \hat Q$.It is the operator of multiplication by the product of
 the function
  $$
   \left(\sqrt
     \frac{(\tilde T-\tilde t-\tilde S+\tilde x)(\tilde t+\tilde N-\tilde x)}{(\tilde t+\tilde N)(\tilde T+\tilde N-\tilde t)}
            +(1/z)\cdot\sqrt\frac{(\tilde S+\tilde N-\tilde x)\tilde x}{(\tilde t+\tilde N)(\tilde T+\tilde N-\tilde t)}\right)^{\hat t-\hat s}
  $$
  and the indicator of the left side of the circle contained between the angles
  $\phi$ and $2\pi-\phi$.

  Finally, we perform the inverse Fourier transform
  and derive the desired correlation kernel:
 \begin{multline*}
    K(\hat x,\hat s; \hat y, \hat t)=\\=
  \frac{1}{2\pi i}\oint_{e^{-i\phi}}^{e^{i\phi}}
   \left(\sqrt
     \frac{(\tilde T-\tilde t-\tilde S+\tilde x)(\tilde t+\tilde N-\tilde x)}{(\tilde t+\tilde N)(\tilde T+\tilde N-\tilde t)}
            +\frac{1}{z}\sqrt\frac{(\tilde S+\tilde N-\tilde x)\tilde x}{(\tilde t+\tilde N)(\tilde T+\tilde N-\tilde t)}\right)^{\hat t-\hat s}
   \cdot \\ \cdot z^{\hat y-\hat x-1}dz\stackrel{w = 1/z}=\\=
   \frac{1}{2\pi i}\oint_{e^{-i\phi}}^{e^{i\phi}}
   \left(\sqrt
     \frac{(\tilde T-\tilde t-\tilde S+\tilde x)(\tilde t+\tilde N-\tilde x)}{(\tilde t+\tilde N)(\tilde T+\tilde N-\tilde t)}
            +w\cdot\sqrt\frac{(\tilde S+\tilde N-\tilde x)\tilde x}{(\tilde t+\tilde N)(\tilde T+\tilde N-\tilde t)}\right)^{\hat t-\hat s}
   \cdot \\ \cdot w^{\hat x-\hat y-1}dw=\\
   =\left(\sqrt
     \frac{(\tilde T-\tilde t-\tilde S+\tilde x)(\tilde t+\tilde N-\tilde x)}{(\tilde t+\tilde N)(\tilde T+\tilde N-\tilde t)}\right)^{\hat t-\hat s}
   \frac{1}{2\pi i}\oint_{e^{-i\phi}}^{e^{i\phi}}
            (1+cw)
   w^{\hat x-\hat y-1}dw
  .\end{multline*}

  The factor outside the integral can be omitted as it corresponds
  to a conjugation of the kernel which doesn't change the correlation functions.
  We obtained the desired expression from the formulation of the theorem.

  In the case $s>t$ reasoning is similar and we will omit some
  details.
  Consider the operator  $\cal G$ which is given by the matrix
  $$
  {\cal G}(x,y)=\begin{cases}
                 K(x,s;y,t),& x\in {\mathfrak X_s}, y\in {\mathfrak X_t},\\
                 0 & \text{for other } x,y.
                \end{cases}
  $$
  It can be expressed in the following way:
  $${\cal G}=V_{s-1}\cdots V_{t+1}\cdot V_{t}\cdot{\cal R}.$$
  Here $V_h$ is the operator defined on the first $N$ functions $f^h_i$
  by the formula $V_h(f^h_i)=\frac{f^{h+1}_i}{c^h_i}$. As above we extended
  it to the operator on the whole space.
  Denote by $\cal R$ the orthogonal projection onto the span of the first $N$ functions $f_i^t$.

  Next we express the operators through the parameters of the limit regime
  and obtain three families of the operators: ${\cal G}_\rho$, $V_{h,\rho}$ and ${\cal R}_\rho$.

  Using the definition of $c^h_i$ we conclude again
  that norms of the operators $V_{h,\rho}$ and ${\cal R}_{\rho}$ are uniformly bounded.

  $\hat{\cal R}$, which is the limit of the operators ${\cal R}_\rho$ when $\rho \to
  \infty$,
  has been already computed in the previous section. This operator is equal to the operator of the restriction
  to the right arc of the unit circle contained between the angles $-\phi$ and $\phi$.

  For further reasoning we need the following lemma:
  \begin{lemma}
   Consider two families $A_\rho$ and $B_\rho$ of the operators in
   the Hilbert space.
   Assume that when  $\rho \to \infty$
   norms of all operators are bounded by some constant,
    $A_\rho\to A$
   and $B_\rho A_\rho\to P$ in the sense of the strong operator
   convergence,  $P$ is an arbitrary orthoprojection  and $E$ is its image.
   Also we assume that $E$ is invariant space for the operator $A$
   and a restriction  $A_E$ of the last operator is invertible.
   Then on the subspace $E$ operators $B_\rho$
   strongly converge to $A_E^{-1}$.
  \end{lemma}
  \begin{proof}
   Consider an arbitrary vector $x$ belonging to the $E$. It follows that one can find
   such $y$ that  $x=Ay$. Then we have:
   $$B_\rho A_\rho y\to y,$$
   $$A_\rho y\to x,$$
   $B_\rho x=B_\rho (A_\rho y+(x-A_\rho y))=B_\rho A_\rho y+B_\rho (x-A_\rho y)\to y+0=y.$
  \end{proof}

  Let us apply the proved lemma taking $A_\rho=U_{h+1,\rho}$ and $B_\rho=V_{h,\rho}$.
  The product $B_\rho A_\rho$ is the operator which is given by the matrix that coincides with
   the kernel studied in the statical case and, thus, converges
   to the projection $\hat{\cal R}$. It was shown above that the operators $A_\rho$
  converge to the operator $\hat U$ which is the operator of multiplication by the function
  $$ \sqrt
     \frac{(\tilde T-\tilde t-\tilde S+\tilde x)(\tilde t+\tilde N-\tilde x)}{(\tilde t+\tilde N)(\tilde T+\tilde N-\tilde t)}
            +(1/z)\cdot\sqrt\frac{(\tilde S+\tilde N-\tilde x)\tilde x}{(\tilde t+\tilde N)(\tilde T+\tilde N-\tilde t)}
  $$
  in the $L_2(S^1)$--realization.

  The restriction of this function to the unit circle is non-zero and invertible when
   $z\neq -1$ while it may be equal to zero when $z=-1$. But the limit
   projection
   $\hat{\cal R}$ is the operator of multiplication by the characteristic function of some
   arc which doesn't contain the point $z=-1$. Consequently, restriction of the operator
    $\hat U$ to the image of projection $\hat{\cal R}$ is invertible and we may apply the last lemma.

  We obtain that (in $L_2(S^1)$--realization)
  limit of the operators ${\cal G}_\rho$ multiplies vectors from the image of the limit
  projection
  $\hat{\cal R}$ by some function. This function is inverse to the function which
  the operator $\hat U$ is multiplying by.
  On the other hand, ${\cal R}_\rho\to 0$ for vectors from the orthogonal
  complement to the image of the projection $\hat{\cal R}$, consequently,
  as norms of all operators $V_{h,\rho}$ are bounded, so ${\cal
  G}_\rho\to 0$.

  Further reasoning is identical with the case $s<t$. Theorem is proved.

 \section{Analysis of the results}
 \subsection{ Frozen regions}

 Now we will take a look more closely at the obtained one-dimensional correlation function which can
also be
 called the density function.

 Recall that all points of the collections of non-intersecting paths under
 consideration form a hexagon whose sides are
 parallel  to one axis, other two ones are parallel to another axis
 and third pair of the sides is inclined 45 degrees with respect to
 the axis.

 In the paper \cite{Cohn-Larsen-Propp} tilings of the hexagon by rhombi were
 studied. These tilings can also be interpreted as collections of non-intersecting
 paths.
 It was proved that so-called ``frozen''(or ``arctic'') regions
 exist.
 In these regions the limit distribution is non-random and tiling consists of rhombi of only one type
 (for the exact wording see Theorem 1.1 in \cite{Cohn-Larsen-Propp}). We will show below that our results
 agree with this theorem and the shape or the frozen regions can be obtained also from
 our dynamical kernel as regions of such values of the parameters of the limit regime that
 local fluctuations disappear and the limit process is trivial (either all points of the lattice
 belong to the random configuration with probability one or no points at all belong to the
 random configuration with probability one).

 Let us fix the parameters of the hexagon $\tilde N$, $\tilde T$, $\tilde S$ in our model
 and study the dependence of the density function on the point $(\tilde t,\tilde x)$.

 \begin{proposition}
  The limit distribution is non-trivial only inside an ellipse,
 which is tangent to the hexagon defined by the parameters
  $\tilde N$, $\tilde T$, $\tilde S$ while outside this ellipse
  the density function is equal to zero or to one.
 \end{proposition}
 \begin{proof}
 Let us compute $K(\hat x,\hat t;\hat x,\hat t)$ which coincides with the density function.
 We have:
 $$K=\frac{1}{2\pi i}\oint_{-\phi}^{\phi}\frac{dw}{w}=\frac{\phi}{\pi}.$$
 For some values of the parameters  either $\phi=0$ or $\phi=\pi$, it means that the density
 is equal to 0 or 1.
 In our terms the desired case is equivalent to the fact that
 in the expression defining the angle  $\phi$ ($\phi= \arccos D$)
 the parameter $D$ is not less than 1 in magnitude.
 We can rewrite it as  $D^2\ge 1$. After simplification last inequality transforms into
 \begin{multline*}
  \tilde T^2\tilde x^2+(\tilde S+\tilde N)^2\tilde t^2+2\tilde x\tilde t
  (\tilde N\tilde T-\tilde S\tilde T-2\tilde S\tilde N)+\\+
  2\tilde t(\tilde S\tilde N^2-\tilde N\tilde T\tilde S-\tilde N^2\tilde T+\tilde S^2\tilde N)
  +2\tilde x(\tilde N\tilde T\tilde S-\tilde N\tilde T^2)
  +\\+\tilde N^2(\tilde T-\tilde S)^2\le 0.
 \end{multline*}

 Interior of some ellipse is given by this expression.
 One can easily check that this ellipse is tangent to the borders of the hexagon under
 consideration
  defined by the parameters
 $\tilde N$, $\tilde T$, $\tilde S$.

 In the paper \cite{Cohn-Larsen-Propp} the frozen region was the same, it was the interior
 of the tangent ellipse.
 \end{proof}

 \subsection{Connection with Okounkov-Reshetikhin model}

 In this section we will show the connection between our results and
 results of the paper \cite{Ok2}.
 In the last paper some point process was studied. It had the correlation kernel equal to
  $$
  K_{OR}(x,s;y,t)=\frac{1}{2\pi i}\oint_{\overline z}^z (1-w)^{t-s} w^{x-y+(s-t)/2-1} dw
 ,$$
  here $z$ is some complex number depending on the parameters of the model. This number is
  less than 1 in magnitude and has positive real component.
  The integration is to be performed over the right arc of the circle when  $t\ge s$, otherwise, over the left arc.
 \begin{proposition}
 One can choose a linear transformation of the lattice  ${\mathbb Z^2}$ in such a way
 that processes given by our kernel and by the kernel from \cite{Ok2}
 are connected by the particles-hole involution, i.e. random configurations given by the former kernel
 are complementary to the ones described by the latter kernel.
 \end{proposition}
 \begin{proof}
  Observe that choosing five parameters of our model we can obtain
  arbitrary values of the kernel parameters $0<c\le 1$ and $0\le\phi\le \pi$.

  Now take a look at the kernel from \cite{Ok2}. Let us change the variables in the integral
  to transform the path of integration into the arc of the unit circle (note that arising factor outside the integral
  can be omitted as this factor doesn't change correlation functions).
  Next, we introduce the linear transformation of the lattice
  $$s'-t' = s - t,$$
  $$x'-y' = x - y + (s-t)/2.$$
  Indeed this transformation maps the lattice considered by Okounkov and Reshetikhin
  into  ${\mathbb Z^2}$.
  The kernel becomes
  $$K(x',s';y',t')=\frac{1}{2\pi i}\oint_{-\phi}^{\phi}(1-cw)^{t'-s'}w^{x'-y'-1}dw
  $$
  (the integration is to be performed over the unit circle, the choice of arc is made in the same way as above,
   $0<c\le 1$)
  .
  For $s\neq t$ we make a change of variables in the integral  $z=-w$:
  $$ K= \frac{(-1)\cdot (-1)^{x-y}}{2\pi i}\oint_{-\phi}^{\phi}(1+cz)^{t'-s'}z^{x'-y'-1}dz
  .$$
  Here the integration is performed over the right arc of the unit circle
  when $t<s$ and over the left one otherwise. The factor $(-1)^{x-y}$ can be omitted again.
  This kernel differs from the kernel of our model by the sign for $s\neq t$.
  While for $s=t$ the kernel differs from the desired one by the arc of the
  integration. To obtain the required arc we should subtract the
  integral from the residue of the integrand at the point 0, i.e. subtract from 1 for $x=y$ and subtract from 0 for $x\neq y$.
  Thus, two kernels are related in the following way: $K_1=\delta_{x=y,s=t}-K_2$.
  But it is exactly the transformation of the kernel in the particle-holes involution.
  (for instance, it was proved in  Appendix A.3 in \cite{BOkOl}).

  Note that when the parameter of the integral $c>1$ in our model we can change the variables: $w=1/z$. The constant
  $c$ becomes $1/c$ and also index of power of $z$ changes from $x-y-1$ into $x-y+(s-t)-1$. However,
  this transform is only another linear transformation of the lattice. Consequently, the proposition
  is proved in that case too.
 \end{proof}

 This result could have been foreseen, because the models are closely related.
 In \cite{Ok2} the ensemble of 3d Young diagrams without any restrictions was studied and one of the interpretation
 of our model is the ensemble of 3d Young diagrams contained in the
 box. Other argument in favor of this result is hypothetical
 universality of the obtained kernel stated in  \cite{Ok2}. One of the additional proofs of the universality
 was derivation of the same kernel in  \cite{Bor} as bulk scaling limit of the cylindric partitions.

\section{Appendix. Possible interpretations of the model under consideration}

The model under consideration has lots of different combinatorial
interpretations. We will show some of them.

Consider a collection of non-intersecting paths we studied through
the paper.

\begin{center}
 {\scalebox{0.7}{\includegraphics{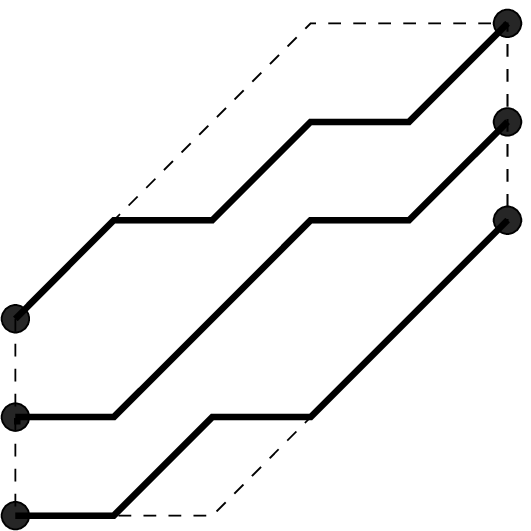}}}

 Figure 2. 3 non-intersecting paths.
\end{center}

Now we incline at an angle of 30 degrees with respect to the
horizontal axis segments of the polygonal lines which were inclined
at an angle of 45 degrees. And replace the horizontal segments by
the segments inclined at an angle of 30 degrees to the axis again,
but this time downward directed. This new picture can be interpreted
as a family of the paths on the surface of 3d Young diagram in the
box:

 \begin{center}
 {\scalebox{0.4}{\includegraphics{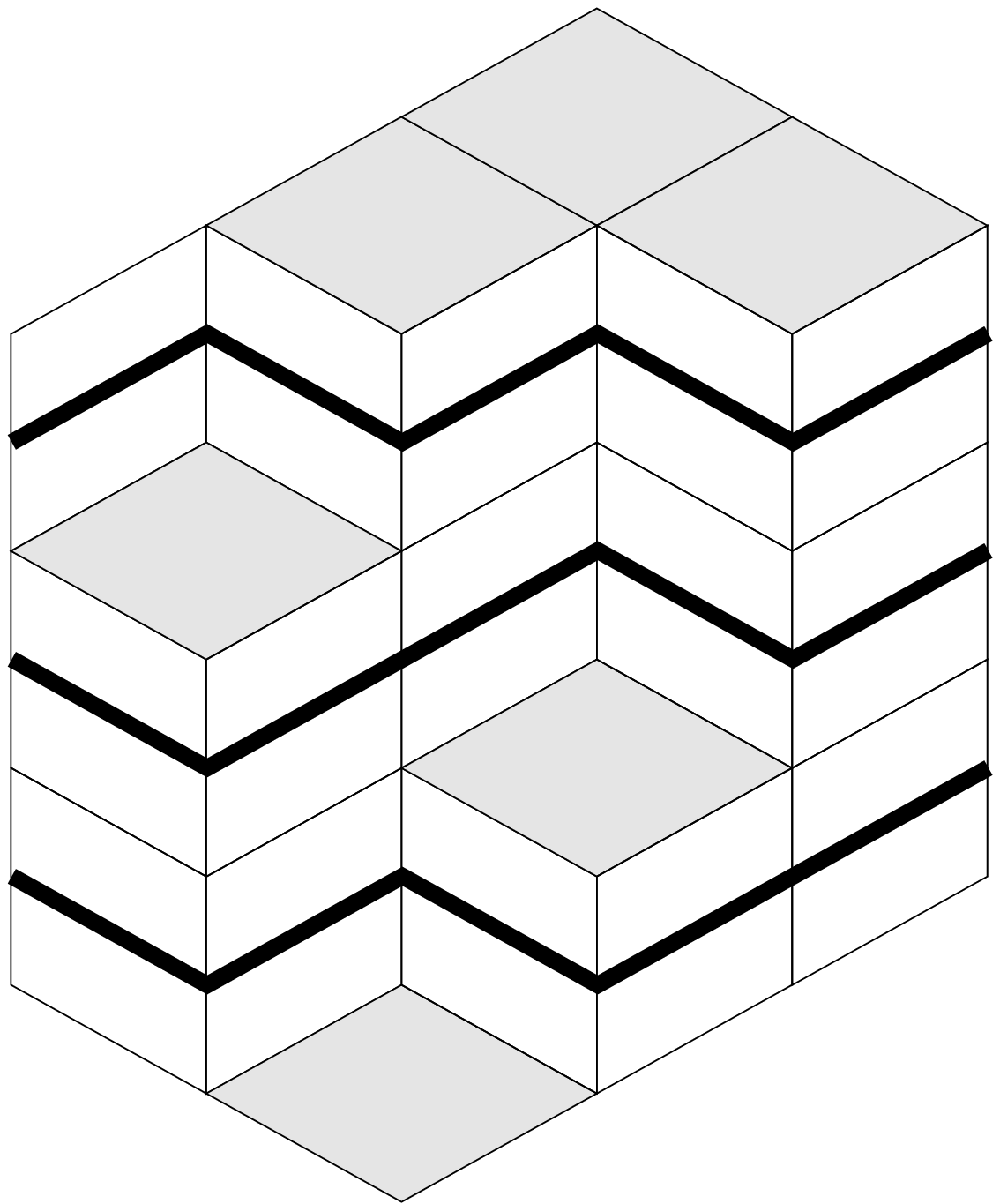}}}

 Figure 3. Non-intersecting paths on the surface of 3d Young diagram.
 \end{center}

The last image can also be regarded as a tiling of the hexagon by
rhombi of 3 types.

\end{document}